\documentclass{amsart}
\usepackage{amsmath,amsfonts,amsthm,amssymb,indentfirst,epic,url,graphics,needspace}

\newtheorem{theorem}{Theorem}

\newtheorem{corollary}[theorem]{Corollary}
\newtheorem{proposition}[theorem]{Proposition}
\newtheorem{definition}[theorem]{Definition}
\newtheorem{question}[theorem]{Question}

\newcommand{\Z}{\mathbb{Z}}
\newcommand{\R}{\mathbb{R}}
\renewcommand{\r}{\mathrm}
\newcommand{\supp}{\r{supp}}

\newcommand{\langl}{\begin{picture}(5.1,10)
\put(1.1,3.3){\rotatebox{60}{\line(1,0){5.5}}}
\put(1.1,3.3){\rotatebox{300}{\line(1,0){5.5}}}
\end{picture}}

\newcommand{\rangl}{\begin{picture}(5,10)
\put(.9,3.3){\rotatebox{120}{\line(1,0){5.5}}}
\put(.9,3.3){\rotatebox{240}{\line(1,0){5.5}}}
\end{picture}}

\begin{document}

\begin{center}
\texttt{
This is the final preprint version of a paper which appeared at \\[.3em]
Communications in Algebra, 49 (2021) 3760-3776. \\[.3em]
The published version is accessible to \\[.3em]
subscribers at https://doi.org/10.1080/00927872.2021.1905825.}
\end{center}
\vspace{2em}

\title%
[Group algebras not made zero by a non-monomial relation]%
{Which group algebras cannot be made zero by imposing\\
a single non-monomial relation?}
\thanks{Readable at
\url{http://math.berkeley.edu/~gbergman/papers/},
archived at \url{https://arXiv.org/abs/1905.12704}\,.
After publication, any updates, errata, related references,
etc., found will be recorded at
\url{http://math.berkeley.edu//~gbergman/papers/}\,.
}


\author{George M. Bergman}
\address{Department of Mathematics\\
University of California\\
Berkeley, CA 94720-3840, USA}
\email{gbergman@math.berkeley.edu}

\begin{abstract}
For which groups $G$ is it true that for all fields $k,$
every non-monomial element of the group algebra $k\,G$
generates a {\em proper} $\!2\!$-sided ideal?
The only groups for which we know this to be true are the
torsion-free abelian groups.
We would, in particular, like to know
whether it is true for all {\em free} groups.

We show that the above property fails for wide classes of groups:
for every group $G$ that contains an element $g\neq 1$ whose image
in $G/[g,G]$ has finite order (in particular, every group
containing a $g\neq 1$ that has finite order, or that
satisfies $g\in [g,G]);$ and for every group
containing an element $g$ which commutes with a distinct conjugate
$h^{-1}gh\neq g$ (in particular, for
every nonabelian solvable group).

Closure properties of the class of
groups satisfying the desired condition are noted.
Further questions are raised.
In particular, a
plausible Freiheitssatz for group algebras of free groups is stated,
which would imply the hoped-for result for such group algebras.
\end{abstract}
\maketitle

\section{Starting point: the question for free groups}\label{S.intro}

Let us give a name to the property we will be considering.

\begin{definition}\label{D.resistant}
A group $G$ will be called {\em resistant} if for every
field $k$ and every element $r=\sum_{g\in G} c_g\,g$ of the group
algebra $k\,G$ whose support, $\supp(r)=\{g\mid c_g\neq 0\},$
has cardinality $>1,$ the $\!2\!$-sided ideal of $k\,G$
generated by $r$ is proper.
\textup{(}Henceforth, ``$2\!$-sided ideal'' will be
shortened to ``ideal''.\textup{)}
\end{definition}

The term `resistant' is sometimes used in the theory
of $\!p\!$-groups with a different meaning \cite{resistant}.

Clearly, every torsion-free abelian group is resistant;
these are the only cases that I know.
The present note arose out of

\begin{question}\label{Q.free}
\cite{overflow}
\cite[Q\,3(i)]{pre-K}
Are all free groups resistant?
\end{question}

This suggests the corresponding
question for groups free in various varieties, such as
free solvable groups of a given derived length,
and free nilpotent groups of given nilpotency class.
If either of those questions had a positive answer (for all derived
lengths, respectively all nilpotency classes), one could
deduce a positive answer to Question~\ref{Q.free}.
But it turns out that both have negative answers, and the easy
arguments proving this show very large classes
of groups to be non-resistant;
we develop these in \S\S\ref{S.binomial}-\ref{S.trinomial} below.
After a few brief digressions in \S\ref{S.variant},
we note in \S\ref{S.univ} some closure properties of the
class of resistant groups;
these imply, inter alia, that if the free group on two
generators is resistant, then so is the free product
of any family of torsion-free abelian groups.
In \S\ref{S.Freiheitssatz} we suggest a plausible
Freiheitssatz for group algebras of
free groups, which, if proved, would imply
a positive answer to Question~\ref{Q.free}.

\section{Which binomial elements generate the improper ideal?}\label{S.binomial}

For $G$ a group and $k$ a field,
elements of $k\,G$ with support of cardinality $1$
(monomials) are invertible, hence are excluded
in the statement of Definition~\ref{D.resistant}.
Given an element with support of cardinality $2$ (a binomial)
$c_1 g_1 + c_2 g_2$ $(c_1, c_2\in k-\{0\},\ g_1\neq g_2\in G)$
we can divide by the unit $c_1 g_2,$ and write
the resulting element as $g-c$ $(g\in G-\{1\},\ c\in k-\{0\}).$

Note that imposing the relation $g-c=0$ on $k\,G$ makes the
image $\overline{g}$ of $g$ central,
so all the conjugates $h^{-1}gh$ of $g\in G$ fall together,
and any relation satisfied in $G$ by these conjugates yields
a relation $\overline{g}^n=1.$
But since $g$ becomes identified with $c,$ if
$g^n$ falls together with $1$ but
$c^n\neq 1$ in $k,$ the resulting factor-algebra collapses.

Theorem~\ref{T.g-c_proper} below characterizes those $g\in G$ such that
this does {\em not} happen for any $c\neq 0$ in any field $k$
(condition~(iv) of that theorem).
Theorem~\ref{T.g-c_improper} gives the details of what happens
for other~$g.$

Writing
\begin{equation}\begin{minipage}[c]{35pc}\label{d.[g,h]}
$[g,\,h] \ = \ g^{-1}\,h^{-1}\,g\,h,$
\end{minipage}\end{equation}
we shall denote by $[g,G]$ the subgroup of $G$
generated by the elements $[g,\,h]$ $(h\in G).$
This is a normal subgroup of $G,$ since
${h'}^{-1}\,[g,h]\,h' = [g,\,h']^{-1}\,[g,\,hh'].$

\begin{theorem}\label{T.g-c_proper}
For an element $g\neq 1$ of a group $G,$ the following conditions
are equivalent.\\[.3em]
\textup{(i)}~~The image $\overline{g}$ of $g$ in $G/[g,G]$
has infinite order.\\[.3em]
\textup{(ii)}~For every equation
$\prod_{i=1}^N (h_i^{-1} g h_i)^{\varepsilon_i}=1$ holding
in $G,$ where $h_i\in G$ and $\varepsilon_i=\pm 1$ for
$i=1,\dots,N,$ one has $\sum_{i=1}^N \varepsilon_i = 0.$\\[.3em]
\textup{(iii)}\,For every equation
$\prod_{i=1}^M h_i^{-1} g h_i= \prod_{j=1}^{M'} {h'}_j^{-1} g h'_j$
holding in $G$ with $h_i,\,h'_j\in G,$ one has $M=M'.$\\[.3em]
\textup{(iv)}~For every field $k$ and element $c\in k-\{0\},$
the element $g-c$ generates a proper ideal
of~$k\,G.$\\[.3em]
\textup{(v)}~~Either for some element $c$ of infinite multiplicative
order in some field $k,$ or for a family of elements $c$ of infinitely
many distinct finite
multiplicative orders in \textup{(}possibly various\textup{)}
fields $k,$ the element\textup{(}s\textup{)} $g-c$ generate proper
ideals in the algebras $k\,G.$
\end{theorem}

\begin{proof}
We shall show (i)$\iff$(ii)$\iff$(iii), and then
(i)$\implies$(iv)$\implies$(v)$\implies$(i).

(i)$\implies$(ii) is immediate, since a relation
$\prod_{i=1}^N (h_i^{-1} g h_i)^{\varepsilon_i}=1$ as in~(ii)
gives $\overline{g}^{\,\sum_{i=1}^N\varepsilon_i}=1$ in $G/[g,G].$
To see the converse, note that a relation $\overline{g}^{\,n}=1$ in
$G/[g,G]$ corresponds to a relation $g^n=\prod_{i=1}^N [g,h_i]^{\pm 1}$
in $G,$ and since each $[g,h_i]$ has the form
$g^{-1}(h_i^{-1}gh_i)^{+1},$
this yields a relation as in~(ii) with the sum of the
exponents equal to $n.$
Hence if~(ii) holds, $\overline{g}^{\,n}=1$ implies $n=0,$ proving~(i).

(ii)$\implies$(iii) is clear, since from any relation as in~(iii),
we get, on right-dividing by the right-hand side,
a relation as in~(ii) with $\sum \varepsilon_i=M-M'.$
To get the converse, note that in any relation as in~(ii), if we
have successive terms
$(h_i^{-1} g h_i)^{-1} (h_{i+1}^{-1} g h_{i+1})^{+1},$ we can
rearrange these as
$(h_{i+1}^{-1} g h_{i+1})^{+1} ({h'_i}^{-1} g h'_i)^{-1},$ where
$h'_i = h_i (h_{i+1}^{-1} g h_{i+1}).$
In this way, we can recursively turn any product as in~(ii)
into a product of the same form, and
with the same $\sum \varepsilon_i,$ in which all the
terms with exponent $-1$ occur to the right of all
terms with exponent $+1.$
That relation is clearly equivalent to a relation
as in~(iii) with $M-M'=\sum \varepsilon_i.$

To get (i)$\implies$(iv), note that assuming~(i),
$k(G/[g,G])$ will be free as a module over its central
Laurent polynomial subalgebra $k\langl\,\overline{g}\,\rangl;$
hence tensoring over $k\langl\,\overline{g}\,\rangl$
with $k\langl\,\overline{g}\,\rangl/(\overline{g}-c)\cong k,$
we get a nonzero homomorphic image of $k\,G$ in
which $g-c$ goes to zero, establishing~(iv).
Clearly, (iv)$\implies$(v).
The proof of (v)$\implies$(i) uses the
idea sketched in the second paragraph of this section:
if we had $\overline{g}^{\,n}=1$ for some $n>1,$ then applying a
case of~(v) in which $c^n\neq 1,$ we would get a contradiction.
\end{proof}

(A special case of the argument for (i)$\implies$(iv), where $G$ is a
free nilpotent group, so that~(i) automatically holds, was pointed
out to me a few years ago by Dave Witte Morris
in a discussion of Question~\ref{Q.free}.)

Turning to elements $g$ not satisfying the above conditions, we have

\begin{theorem}\label{T.g-c_improper}
Let $g\neq 1$ be an element of a group $G$ which does not
satisfy the equivalent conditions of Theorem~\ref{T.g-c_proper}.
Then there exists an integer $n\geq 1$ characterized by the following
equivalent conditions:\\[.3em]
\textup{(i)}~~The image $\overline{g}$ of $g$ in $G/[g,G]$
has order exactly~$n.$\\[.3em]
\textup{(ii)}~On the set of all relations of the form
$\prod_{i=1}^N (h_i^{-1} g h_i)^{\varepsilon_i}=1$ holding in $G,$
the values assumed by $\sum_{i=1}^N \varepsilon_i$
are precisely the multiples of $n.$\\[.3em]
\textup{(iii)}~On the set of all relations of the form
$\prod_{i=1}^M h_i^{-1} g h_i= \prod_{j=1}^{M'} {h'_j}^{-1} g h'_j$
holding in $G,$ the values assumed by $M-M'$
are precisely the multiples of $n.$\\[.3em]
\textup{(iv)}~For $k$ a field and $c$ an element of $k-\{0\},$
the element $g-c$ of~$k\,G$ generates the improper ideal
if and only if $c^n\neq 1.$
\end{theorem}

\begin{proof}[Sketch of proof]
Take $n$ as in~(i).
It is not hard to show that the sets of integers described
in~(ii) and in~(iii) are additive subgroups of $\Z,$
since relations can be ``multiplied'' and ``inverted'';
call their positive generators $n'$ and $n''$ respectively.
Arguments essentially analogous to those in the
proof of Theorem~\ref{T.g-c_proper} show that $n=n'=n''.$
The proof that the $n$ of~(i) satisfies~(iv) is likewise analogous
to the proof of (i)$\iff$(iv) in Theorem~\ref{T.g-c_proper},
and~(iv) uniquely determines $n,$ because for all positive integers $m$
there exist fields with elements of multiplicative order~$m.$
\end{proof}

The next corollary gives some applications of this result.
Case~(b) below, for $\alpha>0,$ demolished my hope that all
$\!2\!$-sided orderable groups might be resistant,
since the group of order-preserving affine endomaps of the real
line is so orderable, by the ordering that makes $g_1\geq g_2$ whenever
$g_1(t)\geq g_2(t)$ in a neighborhood of $+\infty.$
Cases~(c) and~(d) both show that a free product
of free groups with amalgamation of a common
subgroup need not be resistant: the former implies this for
the groups $\langl h, h'\mid h^n = {h'}^n\rangl$ $(n>1),$
the latter, proved with the help of the former, gives the
more general case $\langl h, h'\mid h^n = {h'}^{n'}\rangl$ $(n,\,n'>1).$

Note that where in Theorem~\ref{T.g-c_improper}, $n$ denoted
the least positive integer with various properties, in this
corollary it denotes any such positive integer, i.e., any
positive multiple of the $n$ of that theorem.

\begin{corollary}\label{C.g-c_egs}
For each of the following classes of groups $G,$
elements $g\in G-\{1\},$ and positive integers $n,$ every
element $c\neq 0$ of a field $k$ satisfying
$c^n\neq 1$ has the property that $g-c\in k\,G$
generates the improper ideal.\\[.3em]
\textup{(a)}~~$G$ any group having a nonidentity element~$g$
of finite order, and any $n>1$ such that $g^n=1.$\\[.3em]
\textup{(b)}~~$G$ the group of affine maps of the real
line generated by the elements $g$ and $h,$
where $g(t)=t+1$ and $h(t)=\alpha\,t$ $(t\in \R)$ for
a rational number $\alpha= m/m'\neq 1,0,$ with $n=|m-m'|.$\\[.3em]
\textup{(c)}~~$G$ any group containing elements $h\neq h'$
such that $h^n={h'}^n$ for some $n>1,$ with $g=h^{-1}h'.$\\[.3em]
\textup{(d)}~~$G$ any group containing elements $h_1,$ $h_2$
which do not commute, but such that for some $n>1,$
$h_1$ commutes with $h_2^n,$ with $g=[h_2,h_1].$\vspace{.3em}

Hence no group with any of the indicated properties is resistant.
\end{corollary}

\begin{proof}
We shall denote by ``(i)-(iv)'' the conditions so named
in Theorem~\ref{T.g-c_improper}.
In each case of this corollary, we shall see that one of (i)-(iii) is
satisfied, thus establishing~(iv), which gives the desired conclusion.

In case~(a) it is clear that~(i) is satisfied,
with the $n$ of~(i) some divisor of the given $n.$

In case~(b), note that $h^{-1}\,g\,h$ carries $t$ to $t+(m'/m),$
hence $(h^{-1}\,g\,h)^m = g^{m'},$ so the $n$ of~(iii) will
be a divisor of $|m'-m|.$

In case~(c), for $g$ as defined there, note that in $G/[g,G],$
since $\overline{h}$ must commute with the central element
$\overline{g}= \overline{h}^{\,-1}\,\overline{h'},$
it commutes with $\overline{h'}.$
Hence $\overline{g}^{\,n}=
(\overline{h}^{\,-1}\ \overline{h'})^n=
\overline{h}^{\,-n}\ \overline{h'}^{\,n} = 1,$ so we can apply~(i).

In case~(d), we can apply~(c) above with $h=h_2,$
$h'=h_1^{-1} h_2\, h_1,$ noting that indeed $h\neq h',$ but that
$h^n = h_2^n = h_1^{-1} h_2^n\, h_1$ (by hypothesis), which
equals ${h'}^n.$
\end{proof}

Remarks: \ The final inequality of Theorem~\ref{T.g-c_improper}(iv)
can never hold if $c=1;$ hence that case of
the result says that $g-1$ can never
generate the improper ideal -- which is certainly true,
since $g-1$ lies in the augmentation ideal of $k\,G.$

The case $g=1$ is excluded in the statements of the above
results for a different reason:
in that case, $g-c,$ if not zero, is a monomial
element of $k\,G,$ which we are not
interested in, though the conclusion that $g-c$ generates the improper
ideal unless $c=1$ is true.

Variants of Corollary~\ref{C.g-c_egs}(b)
give further interesting examples.
For instance, though replacing the $\alpha$ of that example
with a transcendental real number will not lead to
a group with the indicated properties, if we take {\em any} real number
$\alpha,$ and let $G$ be generated by the $g$ of that example
together with both $h(t) = \alpha\,t$ and $h'(t) = (\alpha+1)\,t,$
we find that $g (h\,g\,h^{-1}) = h'\,g\,{h'}^{-1},$
so case~(iii) of Theorem~\ref{T.g-c_improper} is again applicable.
If we take $\alpha$ to be the positive root of
the equation $\alpha^2 = 1+\alpha$ (``the golden section''),
then with $G$ again generated just by $g$ and $h,$ we get
$h^2g\,h^{-2} = g\,(h\,g\,h^{-1}),$ with the same consequence.

\section{Some trinomials}\label{S.trinomial}

We do not have a general result on when a trinomial element
of a group algebra generates the improper ideal; but the particular
result given by the next theorem, though technical in its statement,
is easy to prove, and quickly allows us to show that no
nonabelian solvable group is resistant.
The theorem does not involve a choice of
$c\in k,$ so our group algebras could be
taken over any commutative ring, but to keep our context
consistent we will
assume them taken over a specified field $k.$

\begin{theorem}\label{T.trinomial}
Suppose a group $G$ has elements $g,\,h$ such that the normal
subgroup $N\subseteq G$ generated by the commutator $[g,h]$
contains an element $f\neq 1$ which commutes with $g$ in $G.$
Then in the group algebra $k\,G,$ the trinomial $1+h-f$
generates the improper ideal.
\end{theorem}

\begin{proof}
Let us check first that $1+h-f$
is indeed a trinomial; i.e., that $1,$ $h,$ $f$ are distinct.
By hypothesis, $f\neq 1.$
If $h$ were equal to either $1$ or $f,$ it would
commute with $g,$ so the normal subgroup $N$ generated by $[g,h]$
would be trivial, so
since $f\in N$ we would have $f=1,$ again contradicting our hypothesis.

The proof that the ideal generated by $1+h-f$ is improper uses
the same trick:
Modulo that ideal, $h$ falls together with $f-1,$ hence commutes
with $g,$ hence $[g,h]$ becomes $1,$ hence all elements
of $N$ fall together with $1.$
Hence $f-1$ becomes zero, so $h$ becomes zero.
Since $h$ is invertible, our algebra collapses.
\end{proof}

%

In particular,

\begin{corollary}\label{C.trinomial}
If in a group $G$ an element $g$ commutes with a conjugate
$h^{-1}gh\neq g,$ then in
the group algebra $k\,G,$ the trinomial $1+h-[g,h]$
generates the improper ideal.
\end{corollary}

\begin{proof}
If $g$ commutes with $h^{-1}gh\neq g,$ then it also commutes
with $g^{-1}(h^{-1}gh)\neq 1,$ i.e.,
with $[g,h],$ and we can apply the preceding theorem.
\end{proof}

\begin{corollary}\label{C.solvable}
Any nonabelian solvable group $G$ has
an element $g$ which commutes with a conjugate
$h^{-1}gh\neq g;$ hence its group algebra $k\,G$
contains a trinomial element which generates the improper ideal.
\end{corollary}

\begin{proof}
Let $h_1$ and $h_2$ be non-commuting elements of the next-to-last
nontrivial term of the derived series of the solvable group $G.$

If $h_1$ commutes with $h_2^{-1}\,h_1\,h_2,$
then we have the asserted relation, with $g=h_1,$ $h=h_2.$

On the other hand, if $h_1$ does not commute with $h_2^{-1}\,h_1\,h_2,$
then it does not commute with
$h_1^{-1}(h_2^{-1}\,h_1\,h_2) = [h_1,h_2].$
Letting $g = [h_1,h_2]$ and $h=h_1,$ we see that the distinct
elements $g$ and $h^{-1}\,g\,h$ both lie in the last nontrivial term of
the derived series, so they commute with each other, giving
the desired relation.

Corollary~\ref{C.trinomial} gives the final conclusion.
\end{proof}

Though Corollary~\ref{C.trinomial} was aimed at getting this
result, not every group to which that corollary applies need
have a nonabelian solvable subgroup.
Indeed, though the commutativity of $g$ with $h^{-1}g\,h$ implies that
in the sequence of elements $h^{-i} g\,h^i$ $(i\in\Z),$ adjacent
elements commute, this does not force terms which are not adjacent  to
commute, as would be needed to get solvability in an obvious way.
That the universal example of Corollary~\ref{C.trinomial},
$\langl g,h\mid [g,\,h^{-1}gh]=1\rangl,$ contains no nonabelian
solvable subgroup would be lengthy to prove, but I will sketch
an easier example.

Let $F_0$ and $F_1$ be free groups on countably infinite families of
generators, written $(g_{2i})_{i\in\Z}$ and $(g_{2i+1})_{i\in\Z}$
respectively,
let $H$ be an infinite cyclic group $\langl h\rangl,$ and let $G$
be the semidirect product $H\ltimes (F_0\times F_1)$ with
$H$ acting on $F_0\times F_1$ by $h^{-1} g_j h = g_{j+1}.$
It is not hard to verify that the subgroup of $G$ generated by
any two noncommuting elements contains two noncommuting elements
$g,\ g'$ of $F_0\times F_1.$
Such elements must have noncommuting projections either
in $F_0$ or in $F_1;$ but two noncommuting elements of a free
group are free generators of a free subgroup.
Hence any noncommutative subgroup of $G$ contains a free
subgroup on two generators, and so is not solvable.

Curiously, though Corollary~\ref{C.solvable} dashed the hope that
one might be able to prove a positive answer to Question~\ref{Q.free}
by showing that
groups free in appropriate varieties of solvable or nilpotent groups
are resistant, one can nonetheless use the theory of
free nilpotent groups (the $F/F_{i+1}$ in the proof below) to get
a positive result:

\begin{proposition}\label{P.no_trinom}
No free group $F$ contains elements $f,\ g,\ h$ as in
Theorem~\ref{T.trinomial}.
\end{proposition}

\begin{proof}
Suppose a free group $F$ has such elements $f,\ g,\ h.$
Clearly, they will have the same property in some
finitely generated subgroup of $F,$ so since a subgroup
of a free group is again free, we may assume $F$ finitely generated.

Now the subgroup of $F$ generated by $f$ and $g,$ being
commutative but free, has to be free on a single generator $g_0.$
Let $g_0$ lie in the $\!i\!$-th term
of the lower central series of $F$ \cite[\S10.2]{MH}, which
we shall write $F_i,$ but not in $F_{i+1}.$
By \cite[Theorem~11.2.4, p.\,175]{MH}, $F_i/F_{i+1}$ is free abelian,
hence in particular, torsion-free, so $f$ and $g,$ being
nonzero powers of $g_0,$ have nonidentity images in that group;
i.e., neither lies in $F_{i+1}.$

But since $f$ lies in the normal subgroup of $F$
generated by $[g,h],$ a commutator involving $g,$
it {\em must} lie in $F_{i+1}.$
This contradiction completes the proof.
\end{proof}

Let us show using a similar argument that for $F$ a free
group, no {\em binomial}
element of $k\,F$ can generate the improper ideal.
By the observations and results of~\S\ref{S.binomial},
it suffices to prove that if $g\in F-\{1\},$ then its
image $\overline{g}\in F/[g,F]$ has infinite order.
To this end, let $F_i$ be the last term of the lower central series
of $F$ which contains~$g,$ and note that the image of $g$ in
$F_i/F_{i+1}$ has infinite order, since that group is free abelian.
Moreover, $[g,F]\subseteq [F_i,F] = F_{i+1},$ so
the infinite order condition goes over to $\overline{g}\in F/[g,F],$
as required.
Hence

\begin{proposition}\label{P.no_binom}
In a free group, every nonidentity element satisfies the equivalent
conditions of Theorem~\ref{T.g-c_proper}.
Hence in the group algebra of a free group over a field, no
binomial element generates the improper ideal.\qed
\end{proposition}

(Strictly speaking, to call on \cite[Theorem~11.2.4]{MH} we should,
as in the proof of Proposition~\ref{P.no_trinom},
have reduced to the case of $F$ free of finite rank.
That reduction is, again, straightforward.
Alternatively, the theorem cited is easily seen to imply the
corresponding statement for free groups of infinite ranks.)

\section{Some variant conditions}\label{S.variant}

The results of the two preceding sections show that many sorts
of groups $G$ have group algebras in which some non-monomial
element generates the improper ideal.
But cases where group algebras contain non-monomial
{\em invertible} elements are more restricted.
Indeed, until recently it was an open question whether
the group algebra $k\,G$ of a torsion-free group $G$
over a field $k$ could have a non-monomial invertible element, but a
case in which this does occur has recently been found~\cite{GGardam}.
(As noted in the last paragraph of Section~1 of~\cite{GGardam},
the conjecture that no such units existed was
often called ``Kaplansky's unit conjecture'', but
had been raised much earlier; in fact, it occurs
in G.\,Higman's 1940 Ph.D.\ thesis; see
\cite[last sentence of \S{7}]{Sandling}.)

In the opposite direction, if $G$ has torsion elements, it is known
that, with precisely three exceptions, $k\,G$ always has
nontrivial units, though not necessarily of the form $g-c$:
\begin{equation}\begin{minipage}[c]{35pc}\label{d.invertible}
(See \cite[Lemma~13.1.1]{Passman}.)
If $k$ is a field, and $G$ a group having a nonidentity element of
finite order, then $k\,G$ contains a non-monomial invertible element
unless $k\cong \Z/2\Z$ and $G\cong Z_2$ or $Z_3,$
or $k\cong \Z/3\Z$ and $G\cong Z_2.$
In those three cases $k\,G$ has no such element.
\end{minipage}\end{equation}

Returning to general non-monomial elements
which generate the improper ideal, let us note that \eqref{d.invertible}
and Theorem~\ref{T.g-c_improper} between them leave open

\begin{question}\label{Q.invertible}
Suppose $G$ is a group containing an element $g$ which has infinite
order, but whose image in $G/[g,G]$ has finite order $n,$ and that
$k$ is a finite field all of whose nonzero elements $c$ satisfy
$c^n=1$ \textup{(}equivalently,
such that ${\rm card}(k-\{0\})$ divides $n).$
Must $k\,G$ contain a non-monomial element which
generates the improper ideal?
\end{question}

Do the questions we have been studying have interesting extensions
to group rings $D\,G$ over division rings $D$?
In this case, there is no evident analog of ``resistant groups'':
Given any noncommutative division ring $D$ and
any nontrivial group $G,$ if we take any $g\in G-\{1\}$ and
any noncentral element $c\in D,$ say with $ac - ca\neq 0,$
then the $\!2\!$-sided ideal generated by $g-c$ will
be improper, since it contains $a(g-c)-(g-c)a = ac-ca\in D-\{0\}.$
The obvious analog of~\eqref{d.invertible} with
$k$ replaced by a division ring $D$ holds, since if $D$
is not itself a field, it will contain a subfield $k$ properly
larger than its prime subfield,
allowing us to apply~\eqref{d.invertible}.
So I am not aware of any interesting directions for investigation
of group rings over division rings other than fields.

\section{Closure properties of the class of resistant groups}\label{S.univ}

We now return to the main theme of this note.
To examine the class of resistant groups and some
related classes more closely, we make

\begin{definition}[cf.~Definition~\ref{D.resistant}]%
\label{D.k-resistant}
If $G$ is a group and $k$ a field, $G$ will be called
{\em $\!k\!$-resistant} if for every $r\in k\,G$ with support
of cardinality $>1,$ the ideal of $k\,G$
generated by $r$ is proper.

If $\mathcal{K}$ is a class of fields, a group $G$ will be
called  $\!\mathcal{K}\!$-resistant if it is $\!k\!$-resistant
for all $k\in\mathcal{K}.$
\end{definition}

A key fact will be

\begin{proposition}\label{P.univ_G,k}
There exists a set of universal sentences of the form
\begin{equation}\begin{minipage}[c]{35pc}\label{d.univ_G,k}
$(\forall\ g_1,\dots,g_n\in G,\ \forall\ c_1,\dots,c_{n'}\in k)\ %
P(g_1,\dots,g_n,\,c_1,\dots,c_{n'}),$
\end{minipage}\end{equation}
where in each such sentence, $P$ is a Boolean expression in equations in
$g_1,\dots,g_n$ under group operations, and
equations in $c_1,\dots,c_{n'}$ under field operations
\textup{(}with $P,$ $n$ and $n'$ depending on the
sentence~\eqref{d.univ_G,k} in
question\textup{)}, such that for a group $G$ and a field $k,$ $G$
is $\!k\!$-resistant if and only if $G$ and $k$ satisfy all the
conditions~\eqref{d.univ_G,k} in this set.
\end{proposition}

\begin{proof}[Idea of proof]
We want to examine all possible ways that
the ideal of a group algebra
$k\,G$ generated by a non-monomial element
\begin{equation}\begin{minipage}[c]{35pc}\label{d.r=}
$r\ =\ \sum_{i=1}^m c_i\,g_i$
\end{minipage}\end{equation}
$(c_i\in k,$ $g_i\in G,$ $m\geq 2)$
might contain $1\in k\,G,$ and construct a set of
sentences which together say that none of these
possibilities occurs.
The general element of the ideal generated by $r$ has the form
\begin{equation}\begin{minipage}[c]{35pc}\label{d.sum}
$\sum_{j=1}^{m'} c_{m+j}\ g_{m+j}\ r\ g_{m+m'+j}$
\end{minipage}\end{equation}
(where $m$ is as in~\eqref{d.r=},
$m'\geq 0,$ $c_i\in k$ for $m<i\leq m+m',$ and $g_i\in G$ for
$m<i\leq m+2m').$
Clearly, this will equal~$1\in k\,G$ for given
choices of $m',$ $c_{m+1},\dots,c_{m+m'},$
and $g_{m+1},\dots,g_{m+2m'}$ if and only if, when we
expand~\eqref{d.sum} using~\eqref{d.r=}, and note which of the products
\begin{equation}\begin{minipage}[c]{35pc}\label{d.ggg}
$g_{m+j}\ g_i\ g_{m+m'+j}$ $(1\leq i\leq m,\ 1\leq j\leq m')$
\end{minipage}\end{equation}
in the expansion are equal to which others -- an equivalence
relation on product-expressions of this form -- it turns out that the
common value of~\eqref{d.ggg} for one of these
equivalence classes is~$1\in G,$ and the sum of the
coefficients $c_{m+j}\,c_i$ for that equivalence class is $1\in k,$
while for each of the other equivalence classes, the
corresponding sum is~$0.$

So let us construct the sentences~\eqref{d.univ_G,k} as follows.
Each will be determined by a choice of $m,m'\geq 2$
(since only in such cases can~\eqref{d.r=} and~\eqref{d.sum}
describe a non-monomial element $r$ yielding $1$ in the ideal it
generates).
Given these, we let $n=m+2m',$ $n'=m+m',$
and let these index symbols $g_1,\dots,g_n,$ $c_1,\dots,c_{n'}.$
We then list all equivalence relations on the set of
product-expressions~\eqref{d.ggg}, and for each such equivalence
relation, all choices of one distinguished equivalence class.
For each such equivalence relation and distinguished
class, we write down the conjunction of the
list of group-theoretic and
field-theoretic equalities and negations-of-equalities saying that
(i)~$g_1,\dots,g_m$ are distinct,
(ii)~$c_1,\dots,c_{m}$ are nonzero,
(iii)~the products $g_{m+j}\,g_i\,g_{m+m'+j}$ in
each equivalence class are equal, and are unequal
to those in other equivalence classes,
(iv)~for the distinguished equivalence class, the common value
of these products is the group element~$1,$
(v)~the sum, over the distinguished equivalence class,
of the field elements $c_{m+j}\,c_i$ is~$1,$ and
(vi)~the corresponding
sum over each of the other equivalence classes is~$0.$

For each choice
of $m$ and $m',$ we form the conjunction $P,$ over all choices of
equivalence relations and distinguished equivalence classes as above,
of the {\em negations} of the resulting conjunctions of formulas.
The resulting sentence~\eqref{d.univ_G,k},
applied to any group $G$ and field $k,$
says that no $\!m\!$-term element $r\in k\,G$
and $\!m'\!$-term expression~\eqref{d.sum}
constitute a counterexample to the $\!k\!$-resistance of $G.$
Doing this for each choice of $m,m'\geq 2,$ we get the desired
family of sentences~\eqref{d.univ_G,k}.

(Incidentally, not all of the equivalence relations referred
to above will necessarily be consistent with a group structure,
nor need all the systems of equations and negations
of equations in the $\!c\!$'s be
consistent with the properties of fields.
This does not matter: if a conjunction so used in $P$
cannot actually occur, its presence (in negated form,
as above) will have no
effect on the set of $(G,k)$ satisfying~\eqref{d.univ_G,k}.)
\end{proof}

We now modify slightly the form of our conditions~\eqref{d.univ_G,k}.
(And though I have spoken above of ``negations of equalities''
to emphasize that we are working with Boolean expressions in
equalities, I will call them inequalities from here on.)

\begin{proposition}\label{P.univ_vee}
Any sentence~\eqref{d.univ_G,k} as described in the
statement of Proposition~\ref{P.univ_vee} is equivalent to a finite
conjunction of sentences of the form
\begin{equation}\begin{minipage}[c]{35pc}\label{d.univ_vee}
$((\forall\ g_1,\dots,g_n\in G)\ P'(g_1,\dots,g_n))\vee
((\forall\ c_1,\dots,c_{n'}\in k)\ P''(c_1,\dots,c_{n'})),$
\end{minipage}\end{equation}
where each $P'$ is a {\em disjunction} of equations and
inequalities of group-theoretic expressions
in $g_1,\dots,g_n,$ and each $P''$ a disjunction of equations and
inequalities of field-theoretic expressions
in $c_1,\dots,c_{n'}.$
\end{proposition}

\begin{proof}
Any Boolean expression in a set of relations is equivalent
to a conjunction of disjunctions of
families of those relations and their negations.
Applying this to the expression $P(g_1,\dots,g_n,\,c_1,\dots,c_{n'})$
in~\eqref{d.univ_G,k}, and noting that
universal quantification respects conjunctions, we see
that each instance of~\eqref{d.univ_G,k}
is equivalent to the conjunction of a finite family of formulas
$(\forall\ g_1,\dots,g_n\in G,\ \forall\ c_1,\dots,c_{n'}\in k)\ %
P^*(g_1,\dots,g_n,\,c_1,\dots,c_{n'}),$ in each of which
$P^*(g_1,\dots,g_n,\,c_1,\dots,c_{n'})$ is a disjunction
of equations and inequalities.
Each such equation or inequality involves only
the group elements or only the field elements,
so sorting them accordingly, we can rewrite each
$P^*(g_1,\dots,g_n,\,c_1,\dots,c_{n'})$ as
$P'(g_1,\dots,g_n)\vee P''(c_1,\dots,c_{n'}).$
Thus, each sentence~\eqref{d.univ_G,k} is equivalent to a
{\em family} of sentences
\begin{equation}\begin{minipage}[c]{35pc}\label{d.univ_pre-vee}
$(\forall\ g_1,\dots,g_n\in G,\ \forall\ c_1,\dots,c_{n'}\in k)\ %
P'(g_1,\dots,g_n)\vee P''(c_1,\dots,c_{n'}).$
\end{minipage}\end{equation}

Universal quantification does not in general respect disjunction.
(E.g., the statement about
integers, that for all $n,$ either $n$ is even or $n$
is odd, does not entail that either all integers are even
or all are odd.)
But the fact that $P'(g_1,\dots,g_n)$ and $P''(c_1,\dots,c_{n'})$
involve disjoint sets of variables implies that here one {\em can}
pass the disjunction through the quantification, and
rewrite~\eqref{d.univ_pre-vee} as~\eqref{d.univ_vee}.
\end{proof}

We can now move the conditions on field elements entirely
out of our formulas.

\begin{theorem}\label{T.univ_G,K}
For every class $\mathcal{K}$ of fields, there is a
set of sentences of the form
\begin{equation}\begin{minipage}[c]{35pc}\label{d.univ_G,K}
$(\forall\ g_1,\dots,g_n\in G)\ P(g_1,\dots,g_n),$
\end{minipage}\end{equation}
where each $P$ is a disjunction of equations and
inequalities of group-theoretic expressions in $g_1,\dots,g_n,$
such that the $\!\mathcal{K}\!$-resistant groups are
precisely the groups that
satisfy all these sentences~\eqref{d.univ_G,K}.

In particular, taking for $\mathcal{K}$ the class of {\em all} fields,
one gets a set of sentences~\eqref{d.univ_G,K} characterizing the
resistant groups.
\end{theorem}

\begin{proof}
To obtain the desired system of sentences~\eqref{d.univ_G,K},
one goes through the sentences~\eqref{d.univ_vee}
of Proposition~\ref{P.univ_vee}
and asks, for each, whether the second part,
$(\forall\ c_1,\dots,c_{n'}\in k)\ P''(c_1,\dots,c_{n'}),$
holds for {\em all} $k\in\mathcal{K}.$
If it does, we ignore that instance of~\eqref{d.univ_vee},
while if it does not, we include the other part,
$(\forall\ g_1,\dots,g_n\in G)\ P'(g_1,\dots,g_n),$
in our set of sentences~\eqref{d.univ_G,K}.
That the resulting set of sentences characterizes
the $\!\mathcal{K}\!$-resistant groups then follows
from the fact that the sentences~\eqref{d.univ_vee}
characterize pairs $(G,\,k)$ such that $G$ is $\!k\!$-resistant.
\end{proof}

Calling on some standard facts about classes of objects
characterized by universally quantified sentences, this gives us

\begin{corollary}\label{C.closure}
For every class $\mathcal{K}$ of fields \textup{(}including
the class of all fields\textup{)}, the class of
$\!\mathcal{K}\!$-resistant groups\\[.3em]
\textup{(a)} \ is closed under passing to subgroups,\\[.3em]
\textup{(b)} \ is closed under taking direct limits,\\[.3em]
\textup{(c)} \ is closed under taking inverse limits, and\\[.3em]
\textup{(d)} \ is closed under taking ultraproducts.
\end{corollary}

\begin{proof}
(a)-(c) are instances of general results on classes of structures
defined by {\em universal sentences}:
that such classes satisfy~(a) (which is fairly obvious) is the
group-theoretic case of~\cite[Corollary 2.4.2, p.\,49]{WH};
(b)~appears as~\cite[Exercise 5.2.25, p.\,243]{C+K}
(it is only stated there for direct limits over the natural numbers,
but that restriction is not needed);
(c)~appears as~\cite[Theorem~2.4.6]{WH} and
\cite[Exercise 5.2.24, p.\,243]{C+K}.
Property~(d) holds for structures defined by arbitrary first-order
sentences (which can involve arbitrarily long
strings of alternating universal and existential quantifiers); this is
{\L}o{\'s}'s
Theorem \cite[Theorem~2.9,
and Theorem~2.16 (a)$\!\implies\!$(b)]{SB+HPS}.
\end{proof}

Remarks: \ Properties~(b) and~(c) can be deduced from~(a) and~(d):
Given a directed or inversely directed partially ordered set
$(I,\leq)$ and an $\!I\!$-indexed commutative diagram of
groups $(G_i)_{i\in I}$ and group homomorphisms $f_{ij}: G_i\to G_j$
$(i\leq j),$ we consider
the filter on $I$ generated by the principal up-sets, respectively
the principal down-sets.
Letting $\mathcal{U}$ be an ultrafilter on $I$ refining this
filter, we find that the ultraproduct $(\prod_I G_i)/\mathcal{U}$
contains a copy of $\varinjlim_I G_i,$
respectively $\varprojlim_I G_i.$
Hence if each $G_i$ is $\!\mathcal{K}\!$-resistant, that
ultraproduct will be so by~(d), and thus the direct or inverse
limit will be so by~(a).

Here is a nice consequence.

\begin{proposition}\label{P.cP_tf_ab}
If, for some class $\mathcal{K}$ of fields, the free group on two
generators is $\!\mathcal{K}\!$-resistant, then so
is the free product of any family of torsion-free abelian groups.
\end{proposition}

\begin{proof}[Sketch of proof]
We shall see that $\!\mathcal{K}\!$-resistance of each
of the groups or sorts of groups to be listed as~(i)-(vii) below
implies that of the next.
When we refer in our arguments to ``(a)'', ``(b)'', etc.,
these are the four cases of Corollary~\ref{C.closure} above.

Here are the seven (sorts of) groups.
In~(ii)-(vi), $n$ denotes a fixed positive integer, assumed the same
in successive cases; but $\!\mathcal{K}\!$-resistance for
case~(i) will imply the same for case~(ii) for all $n,$
and $\!\mathcal{K}\!$-resistance of~(vi) for all $n$ will be used in
deducing $\!\mathcal{K}\!$-resistance of~(vii).\\[.2em]
\hspace*{2em}(i)~The free group on two generators,\\
\hspace*{2em}(ii)~the free group on $n$ generators,\\
\hspace*{2em}(iii)~a nonprincipal ultrapower over a countable index set
of the free group on $n$ generators,\\
\hspace*{2em}(iv)~the free product of $n$ copies of a nonprincipal
ultrapower of the infinite cyclic group over a
countable index set,\\
\hspace*{2em}(v)~the free product of every family
of $n$ free abelian groups of finite ranks,\\
\hspace*{2em}(vi)~the free product of every family
of $n$ torsion-free abelian groups,\\
\hspace*{2em}(vii)~the free product of every family
(finite or infinite) of torsion-free abelian groups.\vspace{.2em}

Under our assumption that~(i) is $\!\mathcal{K}\!$-resistant,
the fact that the free group on two generators contains
subgroups free on any finite number $n$ of generators
implies the $\!\mathcal{K}\!$-resistance of~(ii), by~(a).
By~(d),~(ii) gives us~(iii).
Using the fact that the free group on $n$ generators is
the free product of $n$ copies of the infinite cyclic group,
and using the normal form for free products of groups, it is
not hard to show that an ultrapower of the free group on $n$ generators
contains the free product of $n$ copies of the corresponding ultrapower
of the infinite cyclic group, giving~(iv).
Now a nonprincipal countable ultrapower of the infinite
cyclic group $Z$ is uncountable,
but is, like $Z,$ torsion-free abelian,
hence it must contain infinitely many linearly independent elements,
hence in particular, it must contain subgroups that are free abelian
of all finite ranks; so the free product of $n$ copies of that group
will contain subgroups as in~(v).
Note next
that every finitely generated subgroup of a torsion-free abelian
group is free abelian of finite rank, hence every torsion-free
abelian group is a direct limit of free abelian groups of finite rank;
hence the corresponding statement holds for free products of $n$ such
groups, so by~(b), (v) gives~(vi).
Similarly, the free product of a general family of
torsion-free abelian groups is a direct limit of
free products of finite families of such groups,
so another  application of~(b) gives~(vii).
\end{proof}

For the group theorist who finds the use of ultrapowers
a bit esoteric, let me indicate, very roughly,
in a different way, what is behind Proposition~\ref{P.cP_tf_ab}.
Suppose we know the free
group $F=\langl x, y\rangl$ is $\!k\!$-resistant, and
want to show the same for, say, $G= \langl x, y,y'\mid [y,y']=1\rangl,$
i.e., the free product of the free abelian group on one generator $x,$
and the free abelian group on two generators $y,\,y'.$
So let $r(x,y,y')$ be a non-monomial element of $k\,G,$ and
consider the group homomorphism $G\to F$ carrying $x$ to $x,$
$y$ to $y,$ and $y'$ to $y^N$ for some large integer $N.$
The key fact (which I leave it to the reader to convince himself
or herself of) is that for $N$ large enough, the image $r(x,y,y^N)$
will also be non-monomial -- in fact, that the distinct terms in the
support of $r(x,y,y')$ will have distinct images in $F.$
Hence by our assumption on $F,$ $r(x,y,y^N)$ generates a proper ideal
of $k\,F;$ hence $r(x,y,y'),$ which maps to it, must
generate a proper ideal of $k\,G.$
The ultrapower argument gets the same conclusion by
showing that the above $G$ is a subgroup of an ultrapower of $F.$

In an appendix,~\S\ref{S.app:_lims} below, we record
some further results on the class of groups that can be obtained
as direct or inverse limits of free products of finite families
of free abelian groups of finite ranks.

Let us ask about a possible result stronger
than that asked for in Question~\ref{Q.free}.

\begin{question}\label{Q.cP}
Is the class of resistant groups closed under taking free products of
groups?
\end{question}

A question which looks more elementary,
but which I also don't know how to approach, is

\begin{question}\label{Q.P}
Is the class of resistant groups closed under taking direct products
of groups?
\end{question}

If we replace ``resistant'' with ``$\!\mathcal{K}\!$-resistant''
for an arbitrary class $\mathcal{K}$ of fields, both questions
have negative answers, as shown by \eqref{d.invertible},
which tells us that
the groups $Z_2$ and $Z_3$ are $\!\mathcal{K}\!$-resistant
for $\mathcal{K}=\{\Z/2\Z\},$
but that this is not true of any groups properly containing them,
though such groups can be formed by taking direct products
or free products of these groups with themselves, with each other, or
with other $\!\mathcal{K}\!$-resistant groups, such as $Z.$
I know of no counterexamples other than these
(up to isomorphism of groups and fields); but these cases
suggest that if Questions~\ref{Q.cP} and~\ref{Q.P} have
positive answers, these are not likely to have trivial proofs.

But here is an easy positive answer to
a special case of Question~\ref{Q.P}.

\begin{proposition}\label{P.GxA}
If $G$ is a resistant group and $A$ a torsion-free abelian
group, then $G\times A$ is again resistant.

More generally, this is true with ``resistant'' replaced by
``$\!\mathcal{K}\!$-resistant'' for any class
$\mathcal{K}$ of fields closed under passing to extension fields.
\end{proposition}

\begin{proof}
We will prove the general statement.
Our proof will use the observation that
$k (G\times A)\cong k\,G\otimes_k k\,A.$

Let $r$ be a non-monomial element of $k (G\times A),$
where $G$ and $A$ are as in the hypothesis, and $k\in\mathcal{K}.$

Suppose first that not all elements of $\r{supp}(r)\subseteq G\times A$
have the same $\!G\!$-component.
Then letting $k'$ be the field of fractions of
the commutative integral domain $k\,A,$ we
see that the natural embedding
\begin{equation}\begin{minipage}[c]{35pc}\label{d.to_k'}
$k (G\times A)\ \cong\ k\,G\otimes_k k\,A
\ \to\ k\,G\otimes_k k'\ \cong\ k'\,G$
\end{minipage}\end{equation}
carries $r$ to a non-monomial element of this $\!k'\!$-algebra.
Since $k'\in\mathcal{K}$ and $G$ is
$\!\mathcal{K}\!$-resistant, that element generates a proper
ideal, whose inverse image in $k(G\times A)$ will be
a proper ideal containing~$r.$

On the other hand, if all elements of $\r{supp}(r)$
have the same $\!G\!$-component $g\in G,$
we can write $r = a\,g$ where $a$ is a non-monomial element of $k\,A.$
Thus, $a$ is non-invertible in the commutative
ring $k\,A,$ and so is contained in a proper maximal ideal $M.$
Letting $k''$ denote the field $(k\,A)/M,$ we see that the natural map
\begin{equation}\begin{minipage}[c]{35pc}\label{d.k''}
$k (G\times A)\ \cong\ k\,G\otimes_k k\,A
\ \to\ k\,G\otimes_k k''\ \cong\ k''\,G$
\end{minipage}\end{equation}
carries $r$ to zero, so the kernel of~\eqref{d.k''}
is a proper ideal containing $r.$
\end{proof}

Returning to Proposition~\ref{P.cP_tf_ab}, the method used to prove
that result can be used to show that still more groups are
$\!\mathcal{K}\!$-resistant under the same assumption.
For instance, let $w$ be any nonidentity element of the
free group on two generators, $\langl x,y\rangl.$
Suppose we take a nonprincipal ultrapower of $\langl x,y\rangl,$
and within it, let $G$ be the subgroup generated by the canonical
image of $\langl x,y\rangl,$ together with an element $z$
of the ultrapower of the free abelian subgroup
$\langl w\rangl,$ such that, in that ultrapower,
$z$ and $w$ are linearly independent.
The centralizer of $z$ in $\langl x,y\rangl$
will be the same as the centralizer of $w\in \langl x,y\rangl;$
and if $w$ is not a proper power therein,
that centralizer will be precisely $\langl w\rangl.$
The group $G$ will then have the
structure $\langl x,y,z\mid [z,w]=1\rangl.$
If $w$ happens to belong to some free generating set $\{v,w\}$
of the free group $\langl x,y\rangl,$ then $G$ is
just the free product of the free abelian groups
on $\{v\}$ and on $\{w,z\},$ and so will have the
form described in Proposition~\ref{P.cP_tf_ab};
but if $w,$ in addition to not being a proper power, does
not belong to such a free generating set
(e.g., if $w= x^m y^n$ with $m,n>1,$ or if
$w= [x,y]),$ then $G$ will lie outside the class of groups
named in that proposition, though it will be resistant by
the reasoning used there.
Another group which one can show resistant in the same way is
$\langl x,\,y,\,z_1,\,z_2,\,z_3 \mid 1=[x,z_1]=[y,z_2]=[xy,z_3]\rangl.$

If we knew that Questions~\ref{Q.cP} and~\ref{Q.P} had positive answers,
then starting with the infinite cyclic group and iterating the
operations of product and free product, we could conclude that a large
class of groups presented by sets of generators together with
commutativity relations among {\em some} pairs of these generators
were resistant.
But this would not cover all groups having presentations of that
description, so let us ask

\begin{question}\label{Q.partial_cm}
Is every {\em right-angled Artin group} -- that is, every group
having, for some set $I$ and some set $J$ of $\!2\!$-element subsets
of $I,$ the presentation
\begin{equation}\begin{minipage}[c]{35pc}\label{d.partial_cm}
$G_{I,J}\ =
\ \langl x_i\ (i\in I)\mid x_i\,x_j=x_j\,x_i$
for all $\{i,j\}\in J\rangl$
\end{minipage}\end{equation}
resistant?
\end{question}

Examples of groups~\eqref{d.partial_cm} that cannot
be gotten by iterating the constructions
of Questions~\ref{Q.cP} and~\ref{Q.P} are those with
$I=\{1,\dots,n\}$ for any $n\geq 4$
and $J=\{\{i,i+1\}\mid 1\leq i<n\},$
and the cyclically indexed analogs of these, with
$I=Z_n$ for any $n\geq 5,$ and $J=\{\{i,i+1\}\mid i\in I\}.$
See~\cite{RAAG} for a general survey of the subject of right-angled
Artin groups.
A positive answer to Question~\ref{Q.partial_cm} would lead,
via the application of ultraproducts and subgroups,
to still further classes of resistant groups.

We remark that right-angled Artin groups are a special case of
{\em Artin-Tits groups}~\cite{A-T}, these being groups
with presentations in which one again starts with a
generating set $\{x_i\mid i\in I\}$ and
a set $J$ of $\!2\!$-element subsets of $I,$ but then imposes
for each $\{i,j\}\in J$ a relation of one of the
forms $(x_i x_j)^{d_{ij}}=(x_j x_i )^{d_{ij}}$ or
$(x_i x_j)^{d_{ij}} x_i =(x_j x_i )^{d_{ij}} x_j,$ with
$d_{ij}\geq 1$ in each case.

However, among such groups $G,$ {\em only} the
right-angled Artin groups, i.e., those whose relations
are all of the first form and have $d_{ij}=1,$ can be resistant.
For if an Artin-Tits group has among its defining relations
a relation $(x_i x_j)^{d_{ij}}=(x_j x_i )^{d_{ij}}$
with $d_{ij}>1,$ then Corollary~\ref{C.g-c_egs}(c) applies
with $h=x_i x_j,$ $h'=x_j x_i,$
while if it has a defining relation
$(x_i x_j)^{d_{ij}} x_i =(x_j x_i )^{d_{ij}} x_j,$
then taking $g=x_i x_j^{-1},$ we see that in $G/[g,G],$
$\overline{x_j}$ commutes with
$\overline{g}=\overline{x_i}\,\overline{x_j}^{-1},$ hence with
$\overline{x_i},$ hence the given defining relation
reduces to $\overline{x_i}=\overline{x_j},$
which says $\overline{g}=1;$ hence Theorem~\ref{T.g-c_proper}(i)
fails, hence so does Theorem~\ref{T.g-c_proper}(iv).
(The above arguments used implicitly the facts that in the former
situation, the relations defining $G$ do not imply $x_i x_j=x_j x_i,$
and that in the latter, they do not imply $x_i=x_j.$
These can be deduced from the result
of~\cite{Artin_monoid} that given a family of relations
as in the definition of an Artin-Tits group, the {\em monoid} that
those relations define embeds in the group that they define,
together with the observation that since each such defining
relation involves the same set of generators on both sides,
any monoid relation $a=b$ that they imply will also
involve the same set of generators on both sides,
and be a consequence of the subset of our defining
relations that involve no generators not occurring in $a$ and $b.)$

\section{A possible Freiheitssatz for group algebras}\label{S.Freiheitssatz}

A standard group-theoretic result, the Freiheitssatz
\cite{Freiheitssatz}, says roughly (we will be
more precise soon) that if one divides the
free group $F$ on generators $(x_i)_{i\in I}$ by the normal subgroup
$N$ generated by a single relator $w,$ and if $x_{i_0}$ is one
of the generators involved in $w,$ then the subgroup of $F/N$
generated by the images of all the other
generators $x_i$ $(i\in I-\{i_0\})$ is free on those generators.
For instance, in the group
$\langl x,y,z\mid x^{-2} y^{-3} x^2 y^3=1\rangl,$
since the relation involves $x,$
the subgroup generated by $y$ and $z$ is free on those generators,
and similarly, since the relation involves $y,$ the elements
$x$ and $z$ generate a free subgroup.
On the other hand, since
$z$ is not involved in our relator $x^{-2} y^{-3} x^2 y^3,$
we cannot say the same about the subgroup generated by $x$ and $y;$
clearly it is not free on $x$ and $y.$

However, one has to be careful about what sort of relator one uses.
The group described above could also be written
$\langl x,y,z\mid z\,(x^{-2} y^{-3} x^2 y^3)\,z^{-1}=1\rangl,$
where the relator now formally involves $z,$ but
the subgroup generated by $x$ and $y$ is clearly still not free.
The Freiheitssatz
excludes such cases by requiring ``cyclically reduced'' relators:
words $w$ in the generators such that not only do symbols $x_i$
and $x_i^{-1}$ never appear adjacent within $w,$ but such
that they also do not
occur, one as the first and the other as the last term of $w.$
Clearly, any relation is equivalent to one given
by a cyclically reduced relator.

Might some sort of Freiheitssatz hold regarding one-relator
factor algebras of {\em group algebras} of free groups?

Suppose $F$ is the free group on generators $(x_i)_{i\in I},$ and
$r\in k\,F$ is a relator we wish to divide out by.
What are the obstacles to hoping that
for any $x_{i_0}$ occurring in $r,$
the subalgebra of our factor algebra generated by
the other $x_i$ is isomorphic to the group algebra on the free group
on those generators?
In this case, there are several.
Note that in the free group on generators $x,y,z,$
the ideal generated by $x^2 - y^3$ is also generated by $z x^2 - z y^3$
and likewise by $x^2 z - y^3 z;$ so we should
forbid any generator-symbol (or inverse of a generator-symbol)
from appearing simultaneously as the first letter of all members
of the support of
our relator, or as the last letter of all these elements.
This condition automatically excludes the
result of conjugating a relator $r$ by a generator
$x_i$ which $r$ does not involve -- unless
the support of $r$ contains the element $1.$
Bringing in that further case, let us make

\begin{definition}\label{D.str_red}
For $F$ the free group on generators $(x_i)_{i\in I},$ and $k$ a field,
we shall call an element $r\in k\,F$ {\em strongly reduced} if, when
the elements of $\r{supp}(r)$ are written in normal form,\\
{\rm (a)} there is no symbol $x_i^{\pm 1}$ with which all
elements of $\r{supp}(r)$ begin,\\
{\rm (b)} there is no symbol $x_i^{\pm 1}$ with which all
elements of $\r{supp}(r)$ end, and\\
{\rm (c)} if $1\in\r{supp}(r)$ \textup{(}so that~{\rm (a)}
and~{\rm (b)} hold trivially\textup{)},
there is no symbol $x_i^{\pm 1}$ such that all elements
of $\r{supp}(r)-\{1\}$ both begin with
$x_i^{\pm 1}$ and end with the inverse symbol, $x_i^{\mp 1}.$
\end{definition}

For any $r\in k\,F,$ let us say that a generator $x_{i_0}$ is
``involved in'' $r$ if $x_{i_0}$ or $x_{i_0}^{-1}$ occurs
anywhere in the normal form of any of the elements of $\r{supp}(r).$
We can now pose

\begin{question}\label{Q.Freiheitssatz}
Let $F$ be the free group on a set $(x_i)_{i\in I},$ $k$ a field,
$r$ a {\em strongly reduced} element of $k\,F,$
and~$(r)\subseteq k\,F$ the ideal it generates.
Then must the following equivalent conditions hold?\\[.2em]
\textup{(i)} \ For every $x_{i_0}$ involved in $r,$
the subalgebra of $k\,F/(r)$ generated by the
images of the $x_i^{\pm 1}$ $(i\in I-\{i_0\})$
is \textup{(}up to natural
isomorphism\textup{)} the group algebra over $k$ of the free
group on $\{x_i\mid i\in I-\{i_0\}\}.$\\[.2em]
\textup{(ii)} \ For every $x_{i_0}$ involved in $r,$
the ideal $(r)$ has zero intersection with the subalgebra
of $k\,F$ generated by $\{x_i^{\pm 1}\mid i \in I-\{i_0\}\}.$\\[.2em]
\textup{(iii)} \ Every nonzero element $s\in (r)$ involves \textup{(}at
least\textup{)} all generators $x_i$ which $r$ involves.
\end{question}

We see from formulation~(ii) above that
a positive answer to this question would imply
a positive answer to Question~\ref{Q.free}.
Formulation~(iii) seems to give the best
insight as to what would be needed to come
up with a proof or a counterexample.

Some further observations:  If we could
prove the suggested Freiheitssatz for elements whose supports contain
$1$ as in condition (c), then the whole statement would follow.
Indeed, given an $r'$ which, rather,
satisfies~(a) and~(b), and whose support does not contain $1,$
let $u$ be an element of minimal length in that support,
and let $v$ and $w$ be (not necessarily distinct)
elements of that support such that the first and last terms in their
normal form expressions differ respectively from the
corresponding terms in the expression for $u.$
The element $r = u^{-1}\,r'$ certainly generates the same
ideal as does $r'$ and has $1$ in its support.
The element $u^{-1}v$ of its support has normal form beginning with
the normal form of $u^{-1},$ hence involves all the
generators that $u$ involves, from which it is not hard
to deduce that $r$ involves exactly the same set of generators as $r'.$
Finally, the initial factor in the normal form
of $u^{-1}v$ is the inverse of the
final factor in that of $u,$ while the final factor
of $u^{-1}w$ is, by assumption, {\em not} the final factor of $u,$
from which it can be seen that $r'$ satisfies~(c); so if
the Freiheitssatz holds for elements as in~(c), the result
for elements as in~(a) and~(b) whose support does not contain~$1$
also holds.

(I claimed in an early preprint of this note that, likewise, if one
knew the result for all elements $r$ satisfying~(a) and~(b)
and not having $1$ in their supports, one could prove
it for elements $r$ as in~(c); but looking more closely, I do not
see how to prove this.
A situation where there is in general no hope of doing this by
applying~(a) and~(b) to a $\!2\!$-sided associate of $r$ is when
$r$ involves only one of our generators; but in that case,
it is not hard to show directly that $r$ satisfies~(i) above.
For two examples where one can indeed get from an $r$ as in~(c)
with $1$ in its support
an element $r'=u\,r\,v$ $(u,v\in F)$ as in~(a) and~(b) not having $1$
in its support -- but not in any
evident systematic way -- consider on the one hand $r=1+x+y+xy,$
and on the other the very similar expression, $r=1+x+y+xy+yx.$
Then $r'=y^{-1} r\,x^{-1}$ works for the former but not the latter;
while $x^{-1} r\,x^{-1}$ works for the latter but not the former.
(Does any such expression work for both?
Yes, it happens that $y^{-1} x^{-1} r\,x^{-1}$ does.)
But I don't see a way to find such $u$ and $v$ for general~$r.)$

Note that a nonzero member of $(r)$ can have
far fewer elements in its support than $r$ itself does.
For a familiar case, take any $g\in F-\{1\}$
let $r = 1+g+\dots+g^{n-1},$ for some large $n,$
and note that $(1-g)r = 1 - g^n$ has only two terms.
For another sort of example, let $x$ and $y$ be two members of the
given free generating set of $F$ and let $r$ have the
form $p(x)+y,$ where $p$ is any polynomial (or more generally, any
Laurent polynomial) in one variable.
Then $x\,r-r\,x = x\,y - y\,x,$ again with only two terms.
Nevertheless, in each of these cases the set of free
{\em generators} involved in the indicated element of $(r)$ still
includes all those involved in $r;$
so they do not give counterexamples to~(iii) above.

Analogs of Questions~\ref{Q.Freiheitssatz} and~\ref{Q.free} for
{\em free associative algebras} (monoid algebras of free monoids) were
posed by P.\,M.\,Cohn in \cite[Conjectures~1 and~3, pp.\,121-122]{PMC}.
Since the free monoid has no invertible elements other than~$1,$
the desired Freiheitssatz did not involve any conditions analogous to
the ``cyclically reduced'' assumption of
the group-theoretic Freiheitssatz, or the ``strongly reduced''
assumption used above.
L.\,G.\,Makar-Limanov~\cite[\S5]{LM-L} subsequently proved
such a Freiheitssatz for $k$ of characteristic~0, and thus
a positive answer to the analog to
Question~~\ref{Q.free} for free associative algebras over such $k.$

\pagebreak
\section{Appendix: direct and inverse limits of free products of free abelian\\ groups of finite rank}\label{S.app:_lims}

We saw in Proposition~\ref{P.cP_tf_ab}
that if the free group on two generators is resistant, then
so is the free product of any family of torsion-free abelian groups.
The resistance, under the above
assumption, of free products of finite families of free abelian
groups of finite ranks was a key step in this deduction, and a key
tool was Corollary~\ref{C.closure}, saying that the class
of resistant groups is closed under subgroups,
direct and inverse limits, and ultraproducts.
We shall show below that for
\begin{equation}\begin{minipage}[c]{35pc}\label{d.calG}
$\mathcal{G}\ =$ the class of all free products
of finite families of free abelian groups of finite ranks,
\end{minipage}\end{equation}
what we can get from $\mathcal{G}$ by direct limits alone
are precisely the groups all of whose finitely
generated subgroups lie in $\mathcal{G},$ while
what we can get by inverse limits alone form (curiously)
a proper subclass of those direct limits.

We begin with a general observation, which we will subsequently apply
to the $\mathcal{G}$ of~\eqref{d.calG}.

\begin{proposition}\label{P.class_G}
Let $\mathcal{G}$ be any class of finitely generated groups which is
closed under isomorphisms, and under passing to finitely generated
subgroups \textup{(}not necessarily
the class described in~\eqref{d.calG}\textup{)}.\\[.3em]
\textup{(a)}\ \ Suppose further that every chain of surjective
non-one-to-one homomorphisms $G_1\to G_2\to\dots$ among members of
$\mathcal{G}$ is finite.
Then the direct limits of directed systems of groups in $\mathcal{G}$
are precisely the groups whose finitely
generated subgroups all belong to $\mathcal{G}.$\\[.3em]
\textup{(b)}\ \ Suppose, rather, that every ``reverse
chain'' of surjective non-one-to-one
homomorphisms, $\dots\to G_{-2}\to G_{-1},$ among members of
$\mathcal{G}$ that admits a common finite bound on the number of
generators required by the $G_i$ is finite.
Then every inverse limit of groups in $\mathcal{G}$
has the property that all its finitely generated subgroups
belong to $\mathcal{G}.$
\textup{(}Hence, such inverse limits are also
direct limits of groups in $\mathcal{G}.)$
\end{proposition}

\begin{proof}
We shall show, under the respective hypotheses of~(a) and~(b), that
if $G$ is a direct limit,
respectively an inverse limit, of groups in $\mathcal{G},$ then
every finitely generated subgroup of $G$ belongs to $\mathcal{G}.$
The fact that every group is the directed union of
its finitely generated subgroups yields
the converse to this result in the
situation of~(a), and hence the parenthetical observation in~(b).

Suppose first that $G$ is the direct limit of a system of
groups $G_i\in\mathcal{G}$ $(i\in I)$ and maps $f_{ij}:G_i\to G_j$
for $i\leq j,$ under a directed partial ordering $\leq$ on $I,$
and let $H$ be a subgroup of $G$ generated by finitely
many elements $g_1,\dots,g_n.$
By the directedness of $I,$ there will be some
$i_0\in I$ such that $G_{i_0}$ contains elements
$g_{i_0,1},\dots,g_{i_0,n}$ mapping to $g_1,\dots,g_n\in G.$
Now consider for each $i\geq i_0$ the subgroup
$H_i\subseteq G_i$ generated
by the images of $g_{i_0,1},\dots,g_{i_0,n}.$
These subgroups will form a directed system of members
of $\mathcal{G}$ (in view of our assumption that
$\mathcal{G}$ is closed under passing to finitely
generated subgroups); so assuming the hypothesis of~(a),
after some point $i_1$ in that directed system, the induced maps
among the $H_i$ will all be one-to-one, hence be isomorphisms.
Hence $H_{i_1}\subseteq G_{i_1}$ is isomorphic to $H\subseteq G;$
moreover, being a finitely generated subgroup of
$G_{i_1}\in\mathcal{G},$ the group $H_{i_1}$ lies in $\mathcal{G},$
so $H\in\mathcal{G},$ as desired.

On the other hand, suppose $G$ is the inverse limit of
a system of groups $G_i\in\mathcal{G}$ $(i\in I)$ and maps $G_i\to G_j$
$(i\leq j)$ under an inversely directed partial ordering on $I,$
and again let $H$ be a finitely generated subgroup of $G.$
This time, we look at the images $H_i\subseteq G_i$ $(i\in I)$
of $H,$ and the induced homomorphisms among these.
If $H$ is generated by $n$ elements, the same will
be true of all the $H_i,$ so by the hypothesis of~(b),
we can find an $i_1$ before which all these
homomorphisms are isomorphisms, and conclude that
$H\cong H_{i_1};$ hence again, $H\in\mathcal{G}.$
\end{proof}

We now turn to the particular $\mathcal{G}$ of~\eqref{d.calG}.

That this $\mathcal{G}$ is closed under passing to
finitely generated subgroups
follows from the Kurosh Subgroup Theorem~\cite[Theorem~7.8]{WD+MJD}.
(That theorem says that any subgroup $H$ of a free product
$G$ of groups is a
free product of copies of subgroups of some of those groups, together
with a free group.
In our situation,
the latter free group, a free product of infinite cyclic groups,
can be regarded simply as bringing additional free abelian groups
of rank~$1$ into the description of $H$ as a free product; so
$H$ is, as required, a member of $\mathcal{G}.)$

Note that if $G\in\mathcal{G}$ is the free product of free abelian
groups $A_1,\dots,A_N$ of ranks $d_1,\dots,d_N \geq 1,$ then it
cannot be generated by fewer than $d_1+\dots+d_N$ elements, since
its abelianization is the free abelian group of that rank.
Moreover, $d_1,\dots,d_N$ determine the structure of $G,$ so there
exist, up to isomorphism, only finitely many groups in $\mathcal{G}$
generated by a given finite number of elements.

Now in the chains of surjective homomorphisms
$G_1\to G_2\to\dots$ and $\dots\to G_{-2}\to G_{-1}$
considered in Proposition~\ref{P.class_G}, the numbers of generators
of the groups in a given chain is always bounded: in the chains
of the former sort, by the number of generators of $G_1;$
in those of the latter sort, by the hypothesis of statement~(b).
Hence if a chain of either sort
in our present $\mathcal{G}$ were infinite, it would
have to contain two isomorphic groups, which would lead to a
non-one-to-one surjective endomorphism of some member of $\mathcal{G}.$
But the groups in $\mathcal{G}$ are known to be
{\em Hopfian}, i.e., not isomorphic to proper homomorphic
images of themselves~\cite{IMSD+HN}.
(That paper shows that a free product of finitely many
finitely generated Hopfian groups is Hopfian.)
Hence the existence of such infinite chains
would lead to a contradiction; so $\mathcal{G}$ satisfies the
hypotheses of both~(a) and~(b).

For this $\mathcal{G},$
the inclusion of the class of inverse limits in the class of
direct limits (noted parenthetically at the end of~(b))
is proper:  The additive group of rational numbers,
though a direct limit of finitely generated free abelian groups,
has no nontrivial homomorphisms into groups in $\mathcal{G},$
hence it is not an inverse limit of such groups.

In summary,

\begin{theorem}\label{t.lims}
For $\mathcal{G}$ as in~\eqref{d.calG},
the direct limits of groups in $\mathcal{G}$
are those groups all of whose finitely
generated subgroups belong to $\mathcal{G},$ while the
inverse limits of groups in $\mathcal{G}$
form a proper subclass thereof.\qed
\end{theorem}

We noted in the proof of Proposition~\ref{P.cP_tf_ab}
that the class of direct limit groups characterized
in the first part of the above theorem
includes all free products of families of torsion-free abelian groups.
A member of that class which is not itself such a free product is
the direct limit of the chain of inclusions of free groups
on two generators
\begin{equation}\begin{minipage}[c]{35pc}\label{d.chain}
$<x_1,x_2>\ \subseteq\ \dots\ \subseteq\ <x_{n-1},x_n>\ \subseteq
\ <x_n,x_{n+1}>\ \subseteq\ \dots,$ where $x_{n-1}~=~[x_n,x_{n+1}].$
\end{minipage}\end{equation}
This group is nontrivial but has trivial abelianization,
which is clearly not the case for any free product of abelian groups.
An example of an {\em inverse} limit of groups in $\mathcal{G}$
which is not a free product of abelian groups
(and so gives another example of a direct limit with this property)
is the inverse limit $G$ of the chain of surjections of free groups
\begin{equation}\begin{minipage}[c]{35pc}\label{d.inv_chain}
$\dots\ \to\ \langl x_1,\dots,x_n\rangl\to\ \dots\ \to
\ \langl x_1,x_2\rangl\ \to\ \langl x_1\rangl,$
\end{minipage}\end{equation}
where each map
$\langl x_1,\dots,x_n\rangl\to \langl x_1,\dots,x_{n-1}\rangl$
takes $x_n$ to $1,$ and fixes the other free generators.
Indeed, one can show that every abelian
subgroup of $G$ is cyclic, so if $G$ were a free product of
abelian groups, the free factors would be infinite cyclic, i.e.,
$G$ would be free; but by~\cite[Corollary to Theorem~1]{GH}
it is non-free.

\section{Acknowledgements}\label{S.Ackn}

I am indebted to Dave Witte Morris for correspondence discussing
Question~\ref{Q.free}, to Warren Dicks, Peter Linnell,
Don Passman, Tom Scanlon and
Eddy Godelle
respectively for helping me with references
and related material used in Theorem~\ref{t.lims},
Proposition~\ref{P.no_trinom}, Section~\ref{S.univ},
Section~\ref{S.variant}, and the paragraph at the end of
Section~\ref{S.univ}; and, finally,
to the referees of several versions of this note
for very helpful comments and pointers to the literature.

\Needspace{4\baselineskip}


\begin{thebibliography}{00}


\bibitem{pre-K} George M. Bergman,
{\em Some questions for possible submission to the next Kourovka
notebook}, unpublished note, 2019, 6\,pp., readable at
\url{https://math.berkeley.edu/~gbergman/papers/unpub/};
earlier version at \url{https://arxiv.org/abs/1904.04298}.

\bibitem{SB+HPS} Stanley Burris and H.\,P.\,Sankappanavar,
{\em A course in universal algebra,}
Graduate Texts in Mathematics, 78. Springer-Verlag,
1981. xvi+276\,pp..
MR0648287

\bibitem{C+K} C. C. Chang and H. J. Keisler,
{\em Model theory,}
Studies in Logic and the Foundations of Mathematics, Vol. 73.
North-Holland Publishing Co.;
American Elsevier Publishing Co.,
1973. xii+550 pp..
MR0409165

\bibitem{RAAG} Ruth Charney,
{\em An introduction to right-angled Artin groups,}
Geom. Dedicata {\bf 125} (2007) 141--158. \
\url{https://arxiv.org/pdf/math/0610668}\,.
MR2322545

\bibitem{PMC} P. M. Cohn,
{\em Progress in free associative algebras,}
Israel J. Math. {\bf 19} (1974) 109--151.
MR0379555

\bibitem{IMSD+HN} I. M. S. Dey and Hanna Neumann,
{\em The Hopf property of free products,}
Math. Z. {\bf 117} (1970) 325--339.
MR0276352


\bibitem{WD+MJD} Warren Dicks and M. J. Dunwoody,
{\em Groups acting on graphs,}
Cambridge Studies in Advanced Mathematics, 17.
Cambridge University Press, 1989.
MR1001965

\bibitem{Dn101} First All-Union Symposium on the Theory of Rings and
Modules (Kishinev, 1968),
{\em The Dniester notebook: Unsolved problems in the theory
of rings and modules,}
Redakc.-Izdat. Otdel Akad. Nauk Moldav. SSR, Kishinev,
1969, 20 pp..
MR0254084

\bibitem{GGardam} Giles Gardam,
{\em A counterexample to the unit conjecture for group rings},
preprint, 7\,pp., 2021.
\url{https://arxiv.org/abs/2102.11818}

\bibitem{A-T} Eddy Godelle and Luis Paris,
{\em Basic questions on Artin-Tits groups,}
pp.299--311 in {\em Configuration Spaces,}
CRM Series, 14, Ed. Norm., Pisa, 2012.
MR3203644

\bibitem{MH} Marshall Hall, Jr.,
{\em The theory of groups,}
The Macmillan Co., 1959. xiii+434\,pp..
MR0103215


\bibitem{GH} Graham Higman,
{\em Unrestricted free products, and varieties of topological groups,}
J. London Math. Soc. {\bf 27} (1952) 73--81.
MR0045730

\bibitem{WH} Wilfrid Hodges,
{\em Model theory,} Encyclopedia of Mathematics and its Applications,
42. Cambridge University Press, 1993. xiv+772 pp..
MR1221741


\bibitem{Freiheitssatz} Wilhelm Magnus,
{\em \"{U}ber diskontinuierliche Gruppen mit einer
definierenden Relation.
\textup{(}Der Freiheitssatz\textup{)}},
J. Reine Angew. Math. {\bf 163} (1930) 141--165.



\bibitem{LM-L} L. G. Makar-Limanov,
{\em Algebraically closed skew fields,}
J. Algebra {\bf 93} (1985) 117--135.
MR0780486

\bibitem{Artin_monoid} Luis Paris,
{\em Artin monoids inject in their groups,}
Comment. Math. Helv. {\bf 77} (2002) 609--637.
MR1933791

\bibitem{Passman} Donald S. Passman,
{\em The algebraic structure of group rings.
Reprint of the 1977 original},
Robert E. Krieger Publishing Co., Inc., Melbourne, FL, 1985.
xiv+734~pp.
MR0798076

\bibitem{Sandling} Robert Sandling,
{\em Graham Higman's thesis ``Units in group rings''},
pp. 93--116 in
{\em Integral representations and applications}
({\em Oberwolfach, 1980}),
Lecture Notes in Math., 882, Springer, Berlin-New York, 1981.
MR0646094
\url{https://www.maths.ed.ac.uk/~v1ranick/papers/higmanthesis.pdf}

\bibitem{resistant} Radu Stancu,
{\em Almost all generalized extraspecial $\!p\!$-groups are resistant,}
J. Algebra {\bf 249} (2002) 120--126.
MR1887988

\bibitem{overflow} mathoverflow,
{\em Must nonunit in group algebra of free group generate proper
two-sided ideal?} \ \url{http://mathoverflow.net/q/198298}\,.
(Submitted to MathOverflow on behalf of the present author by
Dave Witte Morris in 2015.)

\end{thebibliography}
\end{document}